\documentclass[11pt]{article}

\usepackage[margin=1in]{geometry}
\usepackage{amsmath,amssymb,amsthm,mathtools}
\usepackage[hidelinks]{hyperref}
\usepackage{thmtools}
\usepackage{cleveref}
\usepackage{microtype}

\newtheorem{theorem}{Theorem}
\newtheorem{lemma}[theorem]{Lemma}
\newtheorem{proposition}[theorem]{Proposition}

\theoremstyle{definition}
\newtheorem{definition}[theorem]{Definition}
\theoremstyle{remark}

\newcommand{\R}{\mathbb{R}}
\newcommand{\calA}{\mathcal{A}}
\newcommand{\calM}{\mathcal{M}}

\newcommand{\ch}[1]{\chi_{#1}}
\newcommand{\st}{\,:\,}
\newcommand{\dd}{d}

\title{The singleton hypergraph is extremal for the Isolation Lemma}
\author{Vance Faber\footnote{Hoquiam, WA. Email: \texttt{vance.faber@gmail.com}}, David G. Harris\footnote{University of Maryland, Dept. of Computer Science. Email: \texttt{davidgharris29@gmail.com}}}
\date{}

\begin{document}
\maketitle

\begin{abstract}
Let $H$ be an inclusion-free hypergraph on $n$ vertices.  A weight assignment
$w:[n]\to[\dd]$ is isolating if there is a unique edge $e$ whose weight $w(e) = \sum_{i \in e} w(i)$ is minimum. We show that the number of isolating weight assignments is at least
$$
n\sum_{j=0}^{\dd-1} j^{n-1},
$$
a bound which is attained with equality by the  hypergraph consisting of the $n$ singleton edges.  This proves the conjecture stated in Faber \& Harris (2018).

We also prove the bound for a more general class of edge-weight objectives, including arbitrary edge offsets.
\end{abstract}

\section{Introduction}
The Isolation Lemma, introduced by Mulmuley, Vazirani, and Vazirani
\cite{MVV87}, is a basic tool in randomized algorithms.  In one common form, one
has a hypergraph $H$ whose edges encode the possible solutions to a search
problem.  A random weight assignment to the vertices is chosen, and one hopes there is a unique edge of minimum weight.  Such a weight assignment is
called \emph{isolating}.  When this occurs, the original
search problem has been reduced to a unique-solution problem.

The classical Isolation Lemma in \cite{MVV87} gives a lower bound of $\dd^n (1 - n/\dd)$ for the number of isolating weight assignments.  Ta-Shma \cite{TaShma15} improved the bound to $(\dd-1)^n$ via a simple and elegant argument. This was investigated further by Faber \& Harris \cite{FH18}, with a number of improved bounds; in particular, they conjectured that the minimum is attained by the singleton hypergraph, whose
edges are $\{1\},\ldots,\{n\}$. They showed the conjecture for several special
cases as well as an asymptotic regime.

In this note, we fully prove the conjecture of \cite{FH18}, giving the tight extremal bound on the number of isolating assignments.  We will work in a slightly more general setting. Throughout, $H$ is a non-empty inclusion-free hypergraph on vertex set $ [n]=\{1,\ldots,n\}$. For a function $f: H \to \R$ and weight assignment $w: [n] \rightarrow \mathbb Z$, we define the weight of an edge $e$ by
$$
fw(e) := f(e) + \sum_{i \in e} w(i)
$$

For $d \geq 1$ write $[d] = \{1, \dots, d \}$.  The main theorem is the following.

\begin{theorem}\label{thm:main}
Let $\dd, n \geq 1$.  For every $H,f$ as above,  the number of isolating weight assignments $w \in [d]^n$ on $H$ is at least
\[
       n\sum_{j=0}^{\dd-1} j^{n-1}.
\]
This bound is tight for every $\dd,n$, attained by the singleton hypergraph with $f(e) = 0$ for all $e$.
\end{theorem}

(Here and throughout we use the convention $0^0 = 1$ to cover a few edge cases in the formulas.)

\section{Definitions and basic reductions}

We define $\calM_{H,f}(w)$ to be the set of minimum-weight edges, i.e. 
$$
\calM_{H,f}(w) = \{e \in H: fw(e)\leq fw(h) \text{ for every $h\in H$} \}.
$$
The weight assignment $w$ is
\emph{isolating} if $|\mathcal M_{H,f}(w)| = 1$, that is, there is exactly one minimum-weight edge.  In that case, the
unique minimum-weight edge is called the \emph{isolated edge}.   Note that, since $H$ is inclusion-free, the set family $\mathcal M_{H,f}(w)$ is necessarily an antichain.

For a set $S \subseteq [n]$, we define the indicator vector for $S$ by $\chi_S$, i.e. $$
\chi_S(i) = \begin{cases}
1 & \text{for $i \in S$} \\
0 & \text{for $i \notin S$}
\end{cases}
$$

The key insight of \cite{TaShma15} was that an arbitrary weight assignment could be transformed into an isolating assignment by subtracting the indicator of a minimum-weight edge. 

\begin{proposition}[\cite{TaShma15}]
\label{tashma-prop}
Let $e \in \calM_{H,f}(w)$ for a weight assignment $w$. Then the weight assignment
$$
u = w - \chi_e
$$
is isolating, with isolated edge $e$.
\end{proposition}
\begin{proof}
Let $h\in H$ with
$h\neq e$. Then
\begin{align*}
        fu(h) - fu(e) = fw(h) - fw(e) + | e \setminus h |
\end{align*}

If $h \notin \calM_{H,f}(w)$, then $fw(h) - fw(e) > 0$ strictly. If \(h\in\calM_{H,f}(w)\), then \(e\setminus h\neq\emptyset\) since
$H$ is inclusion-free. In all cases, $fu(h) > fu(e)$, and $e$ is the unique minimum edge under $u$.
\end{proof}

Our proof will need an extension of \Cref{tashma-prop}, using the concept of \emph{private pairs} in an antichain.
\begin{definition}
Let $\calA$ be an antichain. We call $(A,S)$ a \emph{private pair} of $\cal A$ if $A \in \cal A$, $S \subseteq A$, and 
\[
        S\nsubseteq B \qquad\text{for every }B\in\calA \setminus \{A \} 
\]
\end{definition}

\begin{proposition}
\label{fh-prop}
 Suppose that $(e,S)$ is a private pair of the antichain $\calM_{H,f}(w)$. Then the weight assignment
$$
u = w - \chi_S
$$
is isolating, with isolated edge $e$.
\end{proposition}
\begin{proof}
 Let $h\in H$ with $h\neq e$.  Since $S\subseteq e$, we have
\begin{equation}
\label{glab}
        fu(h)-fu(e)=fw(h)-fw(e) + |S \setminus h|
\end{equation}

If $h\notin\calM_{H,f}(w)$, then $fw(h)>fw(e)$ strictly.  If $h\in\calM_{H,f}(w)$, then since $(e,S)$ is a private pair of $\calM_{H,f}(w)$, we have $S\setminus h \neq \emptyset$. In all cases
$fu(h)>fu(e)$, and $e$ is the unique minimum edge under $u$.
\end{proof}

 Note that \Cref{tashma-prop} is just the special case of \Cref{fh-prop} with $S = e$.

Following \cite{FH18}, we will partition the isolating assignments by layers, and show an extremal bound on each layer separately. Formally, the \emph{layer} of a weight assignment $w\in[\dd]^n$ is the minimum value of  $w(i)$ over $i \in [n]$.  Define
\begin{align*}
        Z(H,\dd,f) &=\{w\in[\dd]^n \st w \text{ is isolating for }H,f\} \\
        Z_j(H,\dd,f)&=\{w\in Z(H,\dd,f) \st \text{$w$ in layer $j$} \}, \qquad \text{for $j \in [\dd]$}
        \end{align*}

Our main bound will be the following:
\begin{theorem}\label{thm:layer-one}
Let $\dd, n \geq 1$. For every $H, f$, we have
\[
        |Z_1(H,\dd,f)|\geq n(\dd-1)^{n-1}.
\]
\end{theorem}

This bound is tight for each $\dd,n$: for the singleton hypergraph $S_n = \{ \{1 \}, \dots, \{n \} \}$ and the weight function $f \equiv 0$, a weight assignment $w$ in layer one is isolating if and only if it has exactly one coordinate with value one.

This extremal bound, although it only applies to $Z_1$, will allow us to bound the entire set $Z$.

\begin{proof}[Proof of \Cref{thm:main}, given \Cref{thm:layer-one}]
The value for the singleton hypergraph is immediate.  For the lower bound, clearly $$
|Z(H,\dd,f)| = \sum_{j=1}^{\dd} |Z_j(H,\dd,f)|.
$$

For each layer $j$, define
$f'_j:H \to\R$ by
\[
        f'_j(e)=f(e) + (j-1) |e|.
\]

For any layer-one weight assignment $w' \in [\dd-j+1]^n$, define a weight assignment $w \in [\dd]^n$ by
$$
w(i) = w'(i) + j-1
$$

This weight assignment $w$ is in layer $j$, and every such $w$ corresponds to a unique $w'$. Furthermore $fw(e) = f'_j w' (e)$ for all $e$.  So if $w'$ is isolating for $H, f'_j$  then $w$ is isolating for $H,f$. In particular, 
$$
|Z_j(H,\dd,f)| \geq |Z_1(H, \dd-j+1, f'_j)|
$$

By \Cref{thm:layer-one}, this gives
$$
|Z_j(H,\dd,f)| \geq n (\dd-j)^{n-1}
$$
and hence
\[
|Z(H,\dd,f)| \geq \sum_{j=1}^{\dd} n (\dd-j)^{n-1} = n \sum_{j=0}^{\dd-1} j^{n-1}. \qedhere
\]
\end{proof}

This argument is the main reason we need to consider arbitrary edge offsets $f$. It remains to show the key extremal bound \Cref{thm:layer-one}.
\section{A private-pair lemma}

The proof of the layer-one theorem uses some basic counting lemmas for
antichains.

\begin{lemma}
\label{binom-sum}
For any finite sets $T \subseteq A$, there holds
$$
\sum_{S: T \subseteq S \subseteq A} \frac{1}{\binom{|A|}{|S|}} = \frac{|A|+1}{|T|+1}.
$$
\end{lemma}
\begin{proof}
Let \(a=|A|\) and \(t=|T|\).  There are precisely 
\(\binom{a-t}{k-t }\) choices for \(S\) with $|S| = k$.  Hence the left-hand side is
\begin{equation}
\label{yer1}
    \sum_{k=t}^{a}
        \frac{\binom{a-t}{k-t}}{\binom{a}{k}}.
\end{equation}
For \(t\leq k\leq a\), we rewrite the summand as
\[
    \frac{\binom{a-t}{k-t}}{\binom{a}{k}}
    =
    \frac{(a-t)!}{(k-t)! (a-k)!}
    \cdot
    \frac{k! (a-k)!}{a!}  
    =
    \frac{\binom{k}{t}}{\binom{a}{t}}.
    \]

So the sum (\ref{yer1}) is
$$
\frac{1}{\binom{a}{t}} \sum_{k=t}^{a} \binom{k}{t}
$$

By the hockey-stick identity,
\[
\sum_{k=t}^{a} \binom{k}{t} = \binom{a+1}{t+1}
\]

So the overall sum is indeed
\[
\frac{ \binom{a+1}{t+1} }{\binom{a}{t}} = \frac{ (a+1)! }{ (t+1)! (a-t)! } \cdot \frac{ t! (a-t)! }{ a! } = \frac{a+1}{t+1}. \qedhere
\]
\end{proof}

\begin{lemma}\label{lem:private-subset}
Let $\calA$ be a non-empty antichain of finite sets.  Then
\begin{equation}
\label{grub0}
        \sum_{\substack{(A,S) \\\text{private pair of $\mathcal A$}}}
        \frac{1}{(|A|+1) \binom{|A|}{|S|}}
        \geq 1.
\end{equation}
\end{lemma}

\begin{proof}
Let $m=|\calA|$.  Fix $A\in\calA$.  For every $B \in \calA \setminus \{ A \}$,
choose an element $x_B\in A\setminus B$, possible because $\calA$ is an
antichain.  Set
\[
        T_A=\{x_B\st B\in\calA \setminus \{A \} \}.
\]
Then $T_A\subseteq A$, $|T_A|\leq m-1$, and $(A,T_A)$ is a private pair of $\calA$: for every
$B\neq A$, the set $T_A$ contains $x_B\notin B$, and hence $T_A\nsubseteq B$. Every pair $(A,S)$ with $T_A\subseteq S\subseteq A$ is also a private pair of 
$\cal A$. So the contribution from $A$ to the sum in (\ref{grub0}) is at least
\[
       \frac{1}{|A|+1} \sum_{T_A\subseteq S\subseteq A}
        \frac{1}{\binom{|A|}{|S|}}.
\]

By \Cref{binom-sum}, we have
$$
\sum_{T_A\subseteq S\subseteq A}
        \frac{1}{\binom{|A|}{|S|}} = \frac{ |A|+1 }{|T_A|+1}
        $$

So the contribution from
$A$ is at least
\[
        \frac{1}{|T_A|+1}\geq \frac{1}{m}.
\]
Summing this lower bound over the $m$ choices of $A\in\calA$ proves the lemma.
\end{proof}

\section{Proof of \Cref{thm:layer-one}}
Let us fix non-empty inclusion-free $H$ and $\dd, f$. For a weight assignment $w \in [d]^n$, define
$$
I(w) = \{ i \in [n]: w(i) = 1 \}.
$$

We refer to the elements of $I(w)$ as \emph{one-coordinates} of $w$. With this notation, define
\begin{align*}
        X &= Z_1(S_n, \dd, 0) =\{w\in[\dd]^n \st |I(w)| = 1 \}  \\
        U &= Z_1(H,\dd,f) = \{w \in [\dd]^n: \text{$w$ is isolating for $H,f$ and $I(w) \neq \emptyset$} \}
        \end{align*}
        
Then
\[
        |X|=n(\dd-1)^{n-1}.
\]
The goal of this section is to prove that 
$$
|U|\geq |X|.
$$

The proof strategy is to construct non-negative charges $a(w,u)$ for $w\in X$ and
$u\in U$, such that every $w\in X$ sends total charge at least one and
every $u\in U$ receives total charge at most one.  By a double-counting argument, this will imply that $|U|\geq |X|$. It is similar to the proof strategy of \cite{FH18}, which constructed a bipartite graph between $X$ and $U$; the key difference is to allow fractional weights for each pair $a(w,u)$.

\subsection{Constructing the outgoing charges}
We now construct the charges coming out of weight assignments $w \in X$; we will show that
$$
\sum_{u} a(w,u) \geq 1
$$
 
 Take $w\in X$, and let $i$ be the unique element in $I(w)$, and let $\calM(w) = \calM_{H,f}(w)$ be the family of minimum-weight edges under $w$.  There are two cases for the outgoing charge assignment (all charges not specified are set to zero).
 
\subsection*{Case I: some minimum edge avoids $i$}

Choose an arbitrary edge $e\in\calM(w)$ with $i\notin e$.  Set
\[
        u=w-\ch{e},
\]
and set
\[
        a(w,u) := 1.
\]
Since $i\notin e$, the vector
$u$ lies in $[\dd]^n$ and has layer one. By \Cref{tashma-prop}, $u$ is isolating, with isolated edge $e$. Clearly, in this case, $w$ sends out one unit of charge. 

\subsection*{Case II: every minimum edge contains $i$}

Now suppose that every edge in $\calM(w)$ contains $i$. Consider the set family:
\[
        \calA(w)=\{e \setminus \{i \} \st e\in\calM(w)\}
\]

Since $\calM(w)$ is an antichain,  this set family $\calA(w)$ is also an antichain. For each private pair $(e \setminus \{i \},S)$ of $\calA (w)$,  set
\[
        u=w-\ch{S}
\]
and set
\[
        a(w,u) := \frac{1}{ |e|  \binom{|e|-1}{|S|}}.
\]

Consider any such $u$. Since $S \subseteq e \setminus \{i \}$, and $I(w) = \{i \}$, the vector $u$ lies in $[\dd]^n$ and has layer one.  Since every set in \(\calM(w)\) contains \(i\), the pair $(e, S)$ is a private pair of $\calM (w)$.   So \Cref{fh-prop} implies that $u$ is isolating, with isolated edge $e$. 

Any target
\(u=w-\chi_S\) can arise from at most one edge \(e\in\calM(w)\). So the total charge sent from $w$ is 
$$
\sum_{\substack{(e \setminus \{i \}, S)  \\ \text {private pair of $\calA (w)$}}} 
\frac{1}{ |e|  \binom{|e|-1}{|S|}} = 
\sum_{\substack{(e', S)  \\ \text {private pair of $\calA (w)$}}} 
\frac{1}{(|e'|+1) \binom{|e'|}{|S|}}
$$

By Lemma~\ref{lem:private-subset}, this is at least one.

\paragraph{Key observation.} In both cases, when $u$ receives a non-zero charge from $w$, then the isolated edge in $u$ is the chosen edge $e \in \mathcal M(w)$. In Case I, $u$ satisfies
$$
I(u) \not \subseteq e
$$
while in Case II, $u$ satisfies $$
I(u) \subseteq e
$$
since these one-coordinates only come from $S$ and $i$ itself.

\subsection{Bounding the incoming charges}
To complete the proof, it remains to bound the charge received by a fixed right-hand weight. We will show specifically that, for every $u\in U$, we have
\[
        \sum_{w\in X} a(w,u)\leq 1.
\]

So fix $u\in U$, and let $e$ be its isolated edge. 
Note that $I(u) \neq \emptyset$ since $u$ is in layer one. There are again two cases, which we call Case A and Case B to distinguish them from the outgoing cases.

\subsection*{Case A: $I(u) \not \subseteq e$.}
As noted above, $u$ can only receive charge from Case I, and the isolated edge of $u$ is precisely the chosen edge $e$ in $w$. So the preimage of $u$, if it has any, would be forced to be
\[
        w=u+\ch{e}
\]
and so $u$ receives total charge at most one in this situation.

\subsection*{Case B: $I(u) \subseteq e$.}

Consider some $w \in X$ sending charge to $u$. As noted above, this must have been via Case II.  The original unique one-coordinate in the preimage $w$ must be some $i\in I(u)$.  Moreover, if the
decremented set is $S$, then
\[
       I(u) \setminus \{i \} \subseteq  S\subseteq e\setminus\{i\}.
\]
The first containment holds because every vertex in $I(u) \setminus\{i\}$ must have been decremented from weight two in
$w$.

Ignoring all further restrictions on such preimages can only increase the total
possible incoming charge.  Thus the Case II charge received by $u$ is at most
\begin{equation}\label{eq:incoming-sum}
\sum_{i\in I(u)}
\left[
\sum_{I(u)\setminus\{i\}\subseteq
S\subseteq e\setminus\{i\}}
\frac{1}{|e| \binom{|e|-1}{|S|}}
\right].
\end{equation}

By \Cref{binom-sum} applied with $A = e \setminus \{i \}, T = I(u) \setminus \{i \}$, we have
$$
        \sum_{I(u) \setminus \{i \} \subseteq S\subseteq e\setminus\{i\}}
        \frac{1}{ \binom{|e|-1}{|S|}}  = \frac{|e|}{|I(u)|}
        $$
        and hence the total incoming charge to $u$, summed over $i \in I(u)$, is indeed at most one.

\subsection{Putting it together}
At this point, we can finish the proof of \Cref{thm:layer-one}. We have shown that 
\begin{align*}
\sum_{u \in U} a(w,u) &\geq 1 \qquad \text{for all $w \in X$} \\
\sum_{w \in X} a(w,u) &\leq 1 \qquad \text{for all $u \in U$}
\end{align*}

So 
\begin{align*}
|U| &= \sum_{u \in U} 1 \geq \sum_{u \in U} \sum_{w \in X} a(w,u) \\
&= \sum_{w \in X} \sum_{u \in U} a(w,u) \geq \sum_{w \in X} 1  \\
&= |X| = n (\dd-1)^{n-1}.
\end{align*}

This concludes the proof of \Cref{thm:layer-one} and hence \Cref{thm:main}.

\section{Extensions}

Note that \Cref{thm:main} requires the hypergraph $H$ to be inclusion-free. An equivalent formulation is to allow $H$ to be arbitrary and require that $f$ be an inclusion-monotone set function, i.e. $f(e) \leq f(h)$ for $e \subseteq h$. In this case, given nested edges $e \subsetneq h$, the edge $h$ can never be a minimum-weight edge and can simply be removed. The resulting hypergraph becomes inclusion-free automatically.

\medskip

Even more generally, the proof of \Cref{thm:layer-one} uses only two properties of the
objective function. Namely, suppose that $H$ is an arbitrary hypergraph (not necessarily inclusion-free), and for each \(w\in[d]^n, e\in H\)
we have an objective value \(F_w(e)\) in some fixed totally-ordered set. Define
\begin{align*}
\mathcal M_F(w) &=\{e\in H:F_w(e)\leq F_w(h)\text{ for all }h\in H\} \\
        Z(H,\dd,F) &=\{w\in[\dd]^n \st |\mathcal M_F(w)| = 1 \} \\
        Z_j(H,\dd,F)&=\{w\in Z(H,\dd,F) \st \text{$w$ in layer $j$} \}, \qquad \text{for $j \in [\dd]$}
        \end{align*}
        
If we assume that, for every weight assignment $w \in [d]^n$, the two properties hold:
\begin{enumerate}
\item[(A)] $\calM_F(w)$ is an antichain
    \item [(B)] for every private pair $(e,S)$ of $\calM_F(w)$ with $w - \chi_S \in [d]^n$, there holds $\calM_F(w - \chi_S) = \{ e \}.$
\end{enumerate}
then the proof of \Cref{thm:layer-one} applies without change and gives
\[
    |Z_1(H,d,F)|\geq n(d-1)^{n-1},
\]
and, via the shifting argument used in the proof of \Cref{thm:main},
\[
    |Z(H,d,F)|\geq n\sum_{j=0}^{d-1}j^{n-1}.
\]

For example, these properties hold for a function of the form
\[
    F_w(e)=f(e) + \sum_{i\in e} g_i(w(i))
\]
where $H$ is inclusion-free and each \(g_i:[d]\to\mathbb R\) is strictly increasing.  Some applications of the Isolation Lemma (e.g. \cite{Nara}) require such generalized objective functions.

\bigskip

\section*{Acknowledgments}
The proof benefited from ChatGPT's assistance in discovering the fractional
 charging argument.  The authors have independently verified the argument and
 take responsibility for the final text.

\end{document}